\documentclass[11pt]{amsart}

\usepackage{etex,amsfonts, amsmath, amsthm, amssymb, stmaryrd, epsfig, graphics,
  psfrag, latexsym, mathtools, mathrsfs,enumitem, subfig,
  longtable, booktabs, yfonts, centernot,ifthen,eso-pic} 
\usepackage{tikz} \usetikzlibrary{matrix,arrows,shapes,calc}
\tikzstyle{knot}=[thick]
\tikzstyle{knott}=[thick,preaction={draw, line width=4pt, white}]
\tikzstyle{crossing}=[circle,fill=white,inner sep=0,outer sep=0,minimum width=3pt]
\tikzstyle{dot}=[circle,fill=black,inner sep=0,outer sep=0,minimum width=4pt]

\newcommand{\tikzcircle}{\vcenter{\hbox{\begin{tikzpicture}[scale=0.2]\draw[knot] (0,0) circle (1);\end{tikzpicture}}}}
\newcommand{\tikzsaddle}{\vcenter{\hbox{\begin{tikzpicture}[scale=0.4]\draw[knot] (0,0) to[out=45,in=-45] (0,1);\draw[knot] (1,0) to[out=135,in=-135] (1,1);\end{tikzpicture}}}}

\definecolor{darkblue}{rgb}{0,0,0.4} 
 \usepackage[colorlinks=true, citecolor=darkblue, filecolor=darkblue, linkcolor=darkblue,urlcolor=darkblue]{hyperref} 
\usepackage[all]{hypcap}
\usepackage{xr-hyper}
\usepackage[color=blue!20!white,textsize=tiny]{todonotes}

\numberwithin{equation}{section}

\newtheorem{lem}{Lemma}[section]               
\newtheorem{theorem}[lem]{Theorem}
\newtheorem{lemma}[lem]{Lemma}

\newtheorem{corollary}[lem]{Corollary}

\theoremstyle{definition}

\theoremstyle{remark}

\newtheorem{example}[lem]{Example}

\numberwithin{figure}{section}
\numberwithin{table}{section}

\newcommand{\R}{\mathbb{R}}

\newcommand{\QQ}{\mathbb{Q}}

\newcommand{\wt}{\widetilde}
\newcommand{\ol}{\overline}

\newcommand{\from}{\colon}

\newcommand{\too}{\longrightarrow}

\DeclareMathOperator{\Id}{Id}

\DeclareMathOperator{\Tor}{Tor}

\newcommand{\XTorsion}{\mathscr{T}}
\newcommand{\XOrder}{\mathrm{xo}}
\newcommand{\rk}{\mathrm{rk}}

\newcommand{\KhL}{\mathit{Kh}_\mathscr{L}}
\newcommand{\KhCxL}{\mathit{CKh}_\mathscr{L}}

\newcommand{\RR}{\R}

\newcommand{\FF}{\mathbb{F}}

\newcommand*{\defeq}{\mathrel{\vcenter{\baselineskip0.5ex \lineskiplimit0pt
                     \hbox{\scriptsize.}\hbox{\scriptsize.}}}%
                     =}

\begin{document}

 
\title{Ribbon distance and Khovanov homology}

\author{Sucharit Sarkar}
\thanks{\texttt{SS was supported by NSF Grant DMS-1643401.}}
\email{\href{mailto:sucharit@math.ucla.edu}{sucharit@math.ucla.edu}}
\address{Department of Mathematics, University of California, Los Angeles, CA 90095}



\date{}

\begin{abstract}
  We study a notion of distance between knots, defined in terms of the
  number of saddles in ribbon concordances connecting the knots. We
  construct a lower bound on this distance using the $X$-action on
  Lee's perturbation of Khovanov homology.
\end{abstract}
\maketitle
\vspace{-1cm}


\section{Introduction}\label{sec:introduction}

Ever since its inception, Khovanov
homology~\cite{Kho-kh-categorification}, a categorification of the
Jones polynomial, has attracted tremendous interest and has produced
an entire new field of research. It has been generalized in several
orthogonal directions \cite{Bar-kh-tangle-cob, Kho-kh-tangles,
  Lee-kh-endomorphism, KR-kh-matrixfactorizations} and continues to
generate intense activity.  While the primary focus of the field has
been categorification of various low-dimensional topological
invariants---endowing them with new algebraic and higher categorical
structure---it has also produced a small number of stunning
applications in low-dimensional topology as a
by-product. Specifically, Lee's perturbation of Khovanov homology has
been instrumental in producing several applications for knot
cobordisms; the author's personal favorites are Rasmussen's proof of
the Milnor conjecture~\cite{Ras-kh-slice} (bypassing the earlier
gauge-theoretic proof by Kronheimer-Mrowka) and Piccirillo's proof
that the Conway knot is not slice~\cite{Pic-kh-conway}.

This is a very short paper, so we will not burden it with a long
introduction. Let us quickly describe the main results, and proceed
onto the next section.

We define a notion of distance between knots, using the number of
saddles in ribbon concordances connecting the knots. This distance is
finite if and only if the knots are concordant, but it is hard to find
examples of knots arbitrarily large finite distance apart. Using the
$X$-action on Lee's perturbation of Khovanov homology, we construct a
lower bound on this distance, which is the main result of this paper.

\begin{theorem}\label{thm:main}
  If $d$ is the ribbon distance (defined in
  Section~\ref{sec:ribbon-complex}) between knots $K,K'$, then
  \[
  (2X)^d\KhL(K)\cong (2X)^d\KhL(K')
  \] 
  where $\KhL$ is Lee's perturbation of Khovanov homology.
\end{theorem}

In particular, the distance of $K$ from the unknot defines a notion of
complexity for $K$, and it has a lower bound coming from Khovanov
homology. (Coincidentally, this lower bound agrees with the lower
bound on unknotting number from~\cite{AD-kh-unknotting}.)

\begin{corollary}\label{cor:main}
  For any knot $K$, and over any field $\FF$ with $2\neq 0$, the
  extortion order $\XOrder(K)$ (defined in Section~\ref{sec:x-action})
  is a lower bound for the ribbon distance of $K$ from the unknot.
\end{corollary}

\subsection*{Acknowledgment} There has been a sudden abundance of
short cute papers on applications of Khovanov homology to knot
cobordisms, and in particular ribbon concordances
\cite{AD-kh-unknotting,MM-kh-knight,LZ-kh-ribbon}; the present paper
is a result of the author's desire to join the bandsumwagon. Some of
the ideas of this paper are recycled from the above-mentioned papers,
and therefore, he is grateful to their authors. He would also like to
thank Brendan Owens for pointing out some lower bounds for the band
number and Robert Lipshitz for suggesting the wordplay in this
paragraph.

\section{Knot cobordisms}\label{sec:cobordisms}
A \emph{cobordism} from a link $K_0\subset\RR^3\times\{0\}$ to a link
$K_1\subset\RR^3\times\{1\}$ is a properly embedded oriented surface
$F\subset\RR^3\times[0,1]$\footnote{We are working with cobordisms in
  $\RR^3\times[0,1]$ as opposed to the more standard $S^3\times[0,1]$
  for a couple of reasons: naturality of Khovanov cobordism maps has
  only been established in $\RR^3\times[0,1]$ (even up to sign); and a
  subtle sign discrepancy for dotted cobordism maps can be resolved in
  $\RR^3\times[0,1]$.} with boundary the union of $K_1$ and the
orientation-reversal of $K_0$. Call the projection
$\pi_t\from\RR^3\times[0,1]\to[0,1]$ the time function, and assume
$\pi_t|_F$ is Morse; its index $0,1,2$ critical points are called
births, saddles, and deaths. The cobordism may be viewed as a movie as
time runs from $0$ to $1$. For regular values $t$ of $\pi_t|_F$,
$K_t\defeq F\cap(\RR^3\times\{t\})$ is a link; $K_t$ changes by
isotopy with time, with the following local modifications occurring at
births, saddles, and deaths:
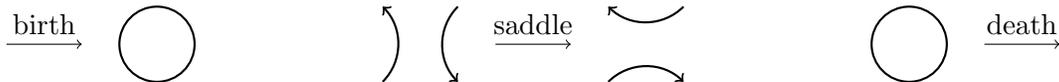
\[
\begin{tikzpicture}
\begin{scope}[xshift=-7cm]
\draw[->] (0,0) to node[above] {birth} (1,0);
\draw[knot] (2,0) circle (0.5);
\end{scope}

\begin{scope}[xshift=0cm]
\draw[knot,->] (-2,-0.5) to[out=45,in=-45] (-2,0.5);
\draw[knot,<-] (-1,-0.5) to[out=135,in=-135] (-1,0.5);
\draw[->] (-0.5,0) to node[above] {saddle} (0.5,0);
\draw[knot,<-] (1,0.5) to[out=-45,in=-135] (2,0.5);
\draw[knot,->] (1,-0.5) to[out=45,in=135] (2,-0.5);
\end{scope}

\begin{scope}[xshift=7cm]
\draw[->] (-1,0) to node[above] {death} (0,0);
\draw[knot] (-2,0) circle (0.5);
\end{scope}

\end{tikzpicture}
\]
We usually work with the projection $\pi_R\from
\RR^2\times\RR\times[0,1]\to\RR^2$, and represent each
$K_t\subset\RR^3\times\{t\}$ by the link diagram $\pi_R(K_t)$. We then
represent the cobordism as a \emph{movie of link diagrams}; usually
$\pi_R(K_t)$ changes by planar isotopy with time, with Reidemeister
moves and the above moves happening at certain time instances (which
by genericity we will assume to be distinct). Two such movies
represent isotopic cobordisms (relative $K_0$ and $K_1$) if and only
if they are related by a sequence of movie
moves~\cite{CRS-knot-knottedsurfaces}.

If $F$ is diffeomorphic to a cylinder, then $F$ is said to be a
\emph{concordance}\footnote{In old literature, the word `cobordism'
  was used instead of `concordance'.} from the knot $K_0$ to the knot
$K_1$. The concordance is said to be \emph{ribbon} if there are no
births~\cite{Gor-top-ribbon}. The famous slice-ribbon conjecture
states that every slice knot has a ribbon concordance to the unknot.

We will also be interested in \emph{dotted cobordisms}, that is,
cobordisms $F$ decorated with finite number of dots in the
interior. We can also represent them by movies of link diagrams,
except now dots are present at certain instances. As before, by
genericity, we will assume these instances are separate from the
births, saddles, deaths, and the Reidemeister moves; moreover, at each
such $t$, the link $K_t$ contains exactly one dot and its projection
to the link diagram $\pi_R(K_t)$ is away from the crossings.

\begin{lemma}\label{lem:dotted-movie-move}
  Assume $F,F'$ are dotted cobordisms (in generic position) with the
  same underlying surface and the same number of dots on each
  component, but differing only in the placement of the dots.  Then
  the movies for $F$ and $F'$ are related a sequence of the following
  movie moves.
  \begin{enumerate}[leftmargin=*]
  \item Far commutation: We may switch the order of the following operations,
    \begin{enumerate}[leftmargin=*]
    \item adding a dot, and then adding another dot;
    \item adding a dot, and then performing a birth;
    \item adding a dot, and then performing a saddle;
    \item performing a death, and then adding a dot;
    \item adding a dot, and then performing a Reidemeister move
      far away.
    \end{enumerate}
  \item Moving dots on link diagrams: If we are adding a dot on one
    side of a crossing on a link diagram $\pi_R(K_t)$, then we can
    instead add it on the other side of the crossing.
  \end{enumerate}
\end{lemma}

\begin{proof}
  We can use the second movie move to move dots freely on $K_t$ for
  each $t$. To move dots in the time direction, we use the first movie
  move, which allows us to move dots past each other, and also past
  births, deaths, saddles, and Reidemeister moves; all the
  possibilities are listed, except the following.
  \begin{enumerate}[leftmargin=*,label=(\alph*)]
  \item Perform a birth, and then add a dot on the newborn
    unknot component. In this case, it is impossible to switch the
    order.
  \item Add a dot to a small unknot component, and then perform a
    death on that component. This is simply the time-reversal of the
    previous case.
  \item Add a dot on some strand of the knot diagram, and then perform
    a Reidemeister move that involves that strand; see below for an
    example with Reidemeister II move.
    
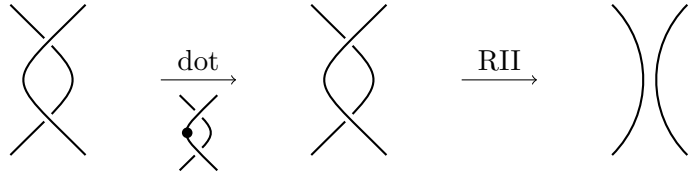
\[
    \begin{tikzpicture}
      \begin{scope}[xshift=-4cm]
        \draw[knott] (0,0) to[looseness=2,out=45,in=-45] (0,2);
        \draw[knott] (1,0) to[looseness=2,out=135,in=-135] (1,2);
        \draw[->] (2,1) -- node[above] {dot} node[anchor=north] {
          \begin{tikzpicture}[scale=0.5]
          \draw[knott,-] (0,0) to[looseness=2,out=45,in=-45] (0,2);
          \draw[knott,-] (1,0) to[looseness=2,out=135,in=-135] (1,2);
          \node[dot,anchor=center] at (0.2,1) {};
          \end{tikzpicture}
        } (3,1);
      \end{scope}
      \begin{scope}
        \draw[knott] (0,0) to[looseness=2,out=45,in=-45] (0,2);
        \draw[knott] (1,0) to[looseness=2,out=135,in=-135] (1,2);
        \draw[->] (2,1) -- node[above] {RII} (3,1);
      \end{scope}
      \begin{scope}[xshift=4cm]
        \draw[knott] (0,0) to[looseness=1,out=45,in=-45] (0,2);
        \draw[knott] (1,0) to[looseness=1,out=135,in=-135] (1,2);
      \end{scope}
    \end{tikzpicture}
    \]
    However, in this case, we may move the dot on the link diagram
    (using the second movie move) away from the strands involved in
    the Reidemeister move, and then use far commutation with the
    Reidemeister move (using the first movie move) to change the
    temporal order of the dot addition and the Reidemeister move.\qedhere
  \end{enumerate}
\end{proof}

\section{Ribbon complexities}\label{sec:ribbon-complex}

There are certain notions of complexities that we can associate to
ribbon concordances. If $K$ is a ribbon knot---that is, if $K$ has a
ribbon concordance to the unknot $U$---then we can define the
\emph{band number} $b(K)$ to be the smallest number of saddles in a
ribbon concordance $K\to U$; this is also the smallest number of bands
if we write $K$ as a band sum of an unlink. This number is usually
called the \emph{ribbon-fusion number} and has lower bounds coming
from the Jones polynomial~\cite{Kan-top-band}; more classically, it is
bounded below by half of $\rk(H_1(\Sigma_K))$---the smallest number of
generators for the first homology of the double branched
cover~\cite{NN-top-ribbon}.

Each ribbon knot bounds a ribbon disk in $\RR^3$, that is, an immersed
disk with only ribbon singularities (as shown with the thick line in
the leftmost figure below). So we may define the \emph{ribbon number}
$r(K)$ to be smallest number of ribbon singularities for ribbon disks
bounding $K$. We may perform saddles near each ribbon singularity (as
shown below) to convert $K$ to an unlink, so $r(K)$ is bounded below
by $b(K)$. A nice argument shows that the knot genus $g(K)$ also
provides a lower bound for $r(K)$~\cite{Fox-top-ribbon}.
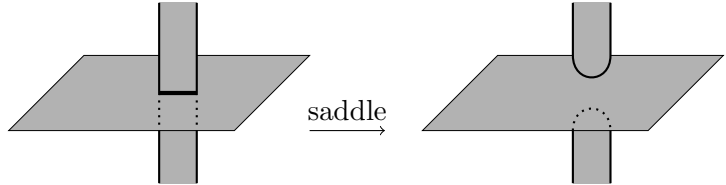
\[
\begin{tikzpicture}
  \begin{scope}

    \fill[black!30] (2,0.5) rectangle (2.5,-0.7);
    \draw[ultra thin, fill=black!30] (0,0) -- (3,0) -- (4,1) --(1,1)--cycle;
    \fill[black!30] (2,1.7) rectangle (2.5,0.5);

    \foreach \i in {2,2.5}{
      \draw[knot] (\i,1.7) -- (\i,0.5);
      \draw[knot] (\i,0) --(\i,-0.7);
      \draw[knot,dotted] (\i,0.5)--(\i,0);
    }
    
    \draw[ultra thick] (2,0.5) --(2.5,0.5);
  \end{scope}
  \draw[->] (4,0) --node[above] {saddle} ++(1,0);
  \begin{scope}[xshift=5.5cm]
    \fill[black!30] (2,0.5) rectangle (2.5,-0.7);
    \draw[ultra thin,fill=black!30] (0,0) -- (3,0) -- (4,1) --(1,1)--cycle;
    \fill[black!30] (2,1.7) rectangle (2.5,0.5);

    \foreach \i in {2,2.5}{
      \draw[knot] (\i,0) --(\i,-0.7);
    }
      \draw[knot] (2,1.7) -- (2,1) to[looseness=2,out=-90,in=-90] (2.5,1) -- (2.5,1.7);
      \draw[knot,dotted] (2,0) to[looseness=2,out=90,in=90] (2.5,0);
        
  \end{scope}
\end{tikzpicture}
\]

We may also define a notion of distance on knots coming from ribbon
concordances. For any two knots $K,K'$, define the \emph{ribbon distance}
$d(K,K')$ to be the smallest $k$ such that there is a sequence of
knots $K=K_0,K_1,\dots,K_{n-1},K_n=K'$ from $K$ to $K'$ and a ribbon
concordance (in some direction) between every consecutive pair
$K_i,K_{i+1}$ with at most $k$ saddles. The following properties are
immediate.
\begin{enumerate}[leftmargin=*]
\item $d(K,K')<\infty$ if and only if $K$ and $K'$ are
  concordant. (For the slightly non-obvious direction, note that if
  $K$ and $K'$ are concordant, then there is some $K''$ with ribbon
  concordances to both $K$ and $K'$, cf.~\cite{Gor-top-ribbon}.)
\item $d(K,K')=0$ if and only if $K$ and $K'$ are isotopic.
\item $d(K,K')=d(K',K)$.
\item $d(K,K'')\leq\max\{d(K,K'),d(K',K'')\}$, and hence $d$
  satisfies the triangle inequality.
\end{enumerate}
This notion of distance complements the more standard notion of
\emph{cobordism distance} which is defined to be the smallest genus of
a cobordism between the two knots. (Cobordism distance between any two
knots is finite, and is zero if and only if the knots are concordant.)

For any slice knot $K$, its distance from the unknot, $d(K,U)$,
therefore provides yet another notion of complexity. It is clear from
the definitions that $d(K,U)\leq b(K)$.

\begin{example}
  Let $K_1$ be the connect sum of the positive and the negative
  trefoil, and let $K_n$ be the connect sum of $n$ copies of $K_1$. We
  have $r(K_1)=g(K_1)=\rk(H_1(\Sigma_{K_1}))=2$ and
  $b(K_1)=d(K_1,U)=1$; $K_1$ can be obtained by adding a band (shown
  by the thick line below) to the $2$-component unlink, which intersects
  the natural disks bounding the unlink in $2$ ribbon singularities.
  
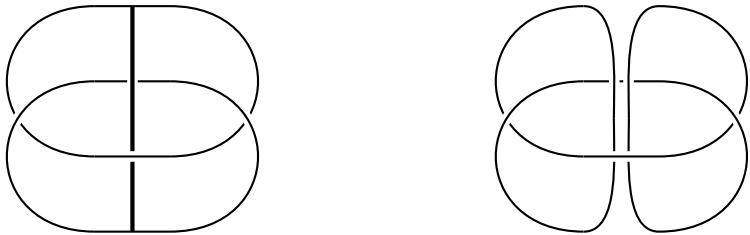
\[
  \begin{tikzpicture}[scale=1]
    \begin{scope}
      \foreach \i [count=\j from 1] in {0,1,2,3}{
        \coordinate (a\j) at (0.5,\i);
        \coordinate (b\j) at (-0.5,\i);
      }
      \coordinate (c) at (0,0);
      \coordinate (d) at (0,3);

      \draw[knott] (a4) to[looseness=2,out=0,in=0] (a2);
      \draw[knott] (a3) to[looseness=2,out=0,in=0] (a1);

      \draw[knott] (b4) to[looseness=2,out=180,in=180] (b2);
      \draw[knott] (b3) to[looseness=2,out=180,in=180] (b1);

      \draw[knott] (b3) -- (a3);
      \draw[knott, ultra thick] (c) -- (d);
      \draw[knot] (b4) -- (a4);
      \draw[knott] (b2) -- (a2);
      \draw[knot] (b1) -- (a1);

    \end{scope}

    \begin{scope}[xshift=6.5cm]
      \foreach \i [count=\j from 1] in {0,1,2,3}{
        \coordinate (a\j) at (0.5,\i);
        \coordinate (b\j) at (-0.5,\i);
      }
      \coordinate (c) at (0,0);
      \coordinate (d) at (0,3);
      \coordinate (a0) at (0.1,1.5);
      \coordinate (b0) at (-0.1,1.5);

      \draw[knott] (a4) to[looseness=2,out=0,in=0] (a2);
      \draw[knott] (a3) to[looseness=2,out=0,in=0] (a1);

      \draw[knott] (b4) to[looseness=2,out=180,in=180] (b2);
      \draw[knott] (b3) to[looseness=2,out=180,in=180] (b1);

      \draw[knott] (a1) to[looseness=0.8,out=180,in=-90] (a0);
      \draw[knott] (b1) to[looseness=0.8,out=0,in=-90] (b0);

      \draw[knott] (b3) -- (a3);

      \draw[knott] (a4) to[looseness=0.8,out=180,in=90] (a0);
      \draw[knott] (b4) to[looseness=0.8,out=0,in=90] (b0);

      \draw[knott] (b2) -- (a2);

    \end{scope}







  \end{tikzpicture}
  \]
  We get $r(K_n)\geq g(K_n)=ng(K_1)=2n$. The ribbon number is
  sub-additive under connect sum, so $r(K_n)\leq n r(K_1)=2n$;
  therefore, $r(K_n)=2n$.  We also get $b(K_n)\geq
  \rk(H_1(\Sigma_{K_n}))/2=\rk(\oplus^n H_1(\Sigma_{K_1}))/2=n$. The band
  number is also sub-additive under connect sum, so $b(K_n)\leq n
  b(K_1)=n$; therefore, $b(K_n)=n$.  Finally $d(K_n,U)=1$ since we
  have a sequence of knots $K_n,K_{n-1},\dots,K_0=U$, and a
  single-saddle ribbon concordance $K_{i+1}\to K_i$ for all $i$,
  obtained by connect summing $K_i$ with the single-saddle ribbon
  concordance $K_1\to U$.
\end{example}

It is unclear if $d(K,U)$ can be arbitrarily large (while staying
finite). In this paper, we will give an example of a knot with
$d(K,U)=2$ (Example~\ref{exam:d-is-2}), and indeed one with $d(K,U)>2$
(Example~\ref{exam:d-is-3}). It is reasonable to guess that the
techniques of this paper, but using knot Floer homology instead of
Khovanov homology, might produce examples of knots with larger values
of $d(K,U)$.

\section{Khovanov homology}\label{sec:khovanov}
Fix a ground ring $R$ and consider the $2$-dimensional Frobenius
algebra $V=R[T][X]/\{X^2=T\}$ over $R[T]$ with comultiplication $V\to
V\otimes_{R[T]} V$ given by
\[
1\mapsto 1\otimes X+X\otimes 1,\qquad X\mapsto X\otimes X+T1\otimes 1.
\]
and counit $V\to R[T]$ given by $1\mapsto 0,X\mapsto 1$. This produces
a Khovanov-style link homology theory~\cite{Kho-kh-categorification}
for any link $K$ by applying it to the Kauffman cube of resolutions of
its link diagram. The resulting theory is usually called the \emph{Lee
  perturbation of Khovanov homology}~\cite{Lee-kh-endomorphism}, and
we will denote it $\KhL(K)$. It is a bigraded homology theory over
$R[T]$ with $R$ in bigrading $(0,0)$ and $T$ in bigrading $(0,-4)$.

A dotted cobordism $F\from K_0\to K_1$ (in generic position) with
$\delta(F)$ dots induces a map 
\[
\KhL(F)\from \KhL(K_0)\to\KhL(K_1)
\]
of $R[T]$-modules of bigrading $(0,\chi(F)-2\delta(F))$
\cite{Kho-kh-cobordism, Bar-kh-tangle-cob, Jac-kh-cobordisms,
  LZ-kh-ribbon}, defined as follows. The movie presentation for $F$ is
a sequence of planar isotopy, Reidemeister moves, births, saddles,
deaths, and dot additions. Except dot addition, each of the other
moves induce a map on $\KhL$ using the Frobenius algebra $V$. The dot
addition map is defined slightly differently. We present a careful
definition below that avoids a sign issue.

An \emph{elementary dotted cobordism} from $K\to K$ is a product
cobordism decorated with a single dot. Consider (the projection of)
the dot on the oriented link diagram $\pi_R(K)$. Checkerboard color
the complement of the link diagram in $\RR^2$ so that the unbounded
region is colored white. If the arc in the link diagram containing the
dot is oriented as the boundary of a black region, define the sign of
the dot to be $(+1)$, otherwise, define it to be $(-1)$. Then define
the dotted cobordism map $\KhL(K)\to\KhL(K)$ to be the map merging a
small unknot labeled $X$ near the dot, times the sign of the dot.

It is well-known that two isotopic (rel boundary) undotted knot
cobordisms induce the same map $\KhL(K_0)\to\KhL(K_1)$, up to an
overall sign.\footnote{This sign issue can also be
  resolved~\cite{CMW-kh-functoriality}, but we will not need to.} We
have a similar variant for dotted cobordisms.

\begin{lemma}\label{lem:dotted-map-same}
  Assume $F,F'$ are dotted cobordisms (in generic position) with the
  same underlying surface and the same number of dots on each
  component, but differing only in the placement of the dots. Then
  they induce the same map on $\KhL$, including the sign.
\end{lemma}

\begin{proof}
  We merely have to check that the map is unchanged under the movie
  moves listed in Lemma~\ref{lem:dotted-movie-move}. The first movie
  move (far commutation) is clear. For the second movie move (moving
  the dot past a crossing), we may check directly that on the Khovanov
  chain complex level, the map associated to merging a small unknot
  labeled $X$ to some strand is homotopic to \emph{negative} of the
  map associated to merging a small unknot labeled $X$ to the
  corresponding strand on the opposite side of a crossing,
  cf.~\cite{BLS-kh-pointed}. Therefore we have the same map on
  homology for dot addition on either side of a crossing.
\end{proof}

The main advantage of using dotted cobordisms is the famous
\emph{neck-cutting relation}. We will need it in the following two
forms.

\begin{lemma}\label{lem:neck-cutting-1}
  Assume the link diagram $\pi_R(K)$ for $K$ contains a small unknot
  $U$. Then, up to an overall sign, the identity map
  $\KhL(K)\to\KhL(K)$ is the sum of the following two maps,
  \begin{enumerate}[leftmargin=*]
  \item add a dot to $U$, perform a death on $U$, and perform a
    rebirth for $U$;
    \item perform a death on $U$, perform a rebirth for $U$, and add a
      dot to $U$.
  \end{enumerate}
  In terms of movies,
  \[
  \pm\KhL(\tikzcircle\stackrel{\Id}{\too}\tikzcircle)=
  \KhL(\tikzcircle\stackrel{\mathrm{dot}}{\too}\tikzcircle\stackrel{\mathrm{death}}{\too}\qquad\stackrel{\mathrm{birth}}{\too}\tikzcircle)
  +
  \KhL(\tikzcircle\stackrel{\mathrm{death}}{\too}\qquad\stackrel{\mathrm{birth}}{\too}\tikzcircle\stackrel{\mathrm{dot}}{\too}\tikzcircle).
  \]
\end{lemma}

\begin{proof}
  If the dot addition maps are given by merging small unknots labeled
  $X$, then it is easy to check that the above equation holds (without
  the sign) on the nose at the Khovanov chain complex level. However,
  the actual dot addition map has an extra sign given by the sign of
  the dot. But the unknot $U$ before death and the unknot $U$ after
  birth are oriented in the same way, so the two dots have the same
  sign, and consequently, the above equation holds up to an overall
  sign.
\end{proof}

\begin{lemma}\label{lem:neck-cutting-2}
  Assume $F\from K\to K$ is a cobordism obtained by performing an
  elementary saddle on the link diagram $\pi_R(K)$ for $K$, followed
  by performing the saddle in reverse. Then, up to an overall sign,
  the map $\KhL(K)\to\KhL(K)$ is the sum of the following two maps,
  \begin{enumerate}[leftmargin=*]
  \item add a dot to one of the two strands in $\pi_R(K)$ involved in
    the saddle;
  \item add a dot to the other strand in $\pi_R(K)$ involved in the
    saddle.
  \end{enumerate}
  In terms of movies,
  \[
  \pm\KhL(\,\,\tikzsaddle\stackrel{\mathrm{saddle}}{\too}\rotatebox[origin=c]{90}{$\tikzsaddle$}\stackrel{\mathrm{saddle}}{\too}\tikzsaddle\,\,)=
  \KhL(\,\,\tikzsaddle\xrightarrow[%
  \vcenter{\hbox{\begin{tikzpicture}[scale=0.3]\draw[knot] (0,0) to[out=45,in=-45] (0,1);\draw[knot] (1,0) to[out=135,in=-135] (1,1);\node[dot] at (0.2,0.5) {};\end{tikzpicture}}}%
  ]{\mathrm{dot}}\tikzsaddle\,\,)
  +
  \KhL(\,\,\tikzsaddle\xrightarrow[%
  \vcenter{\hbox{\begin{tikzpicture}[scale=0.3]\draw[knot] (0,0) to[out=45,in=-45] (0,1);\draw[knot] (1,0) to[out=135,in=-135] (1,1);\node[dot] at (1-0.2,0.5) {};\end{tikzpicture}}}%
  ]{\mathrm{dot}}\tikzsaddle\,\,).
  \]
\end{lemma}

\begin{proof}
  The proof is very similar to the previous proof. If the dot addition
  maps are given by merging small unknots labeled $X$, then the
  equation holds (without the sign) at the Khovanov chain complex
  level. However, since the saddle is an oriented saddle, the two dots
  on the two strands have the same sign, and consequently, the above
  equation holds up to an overall sign.
\end{proof}

\section{$X$-action on Khovanov homology}\label{sec:x-action}
If we fix a component of $K$, then the map $\KhL(K)\to\KhL(K)$
associated to the elementary dotted cobordism $K\to K$ that has a
single dot on the chosen component is denoted $X$ (since it comes from
merging an unknot labeled $X$), and is often called the
\emph{$X$-action} on $\KhL(K)$. It is clear from the Frobenius algebra
$V$ that $X^2=T$.  This makes $\KhL(K)$ a module over
$R[T,X]/\{X^2=T\}=R[X]$ (although the module structure depends on
chosen link component).

Khovanov homology of connect sums has a nice expression using the
$X$-actions.
\begin{lemma}\label{lem:khovanov-connect-sum}
  Let $K,K'$ be links with chosen components, and let $K\# K'$ be the
  link obtained by connect summing the chosen components. Then 
  \[
  \KhL(K\# K')\cong\Sigma^{0,-1}\Tor_{R[X]}(\KhL(K),\KhL(K'))
  \]
  as bigraded $R[X]$-modules, with the $X$-action on the right-hand
  side induced from the $X$-action on either $\KhL(K)$ or
  $\KhL(K')$. Here $\Sigma^{a,b}$ denotes an upward bigrading shift by
  $(a,b)$, that is, tensoring with a single $R$ in bigrading $(a,b)$.
\end{lemma}

\begin{proof}
  The argument entirely follows Khovanov's argument for his original
  invariant (which is the specialization $X^2=T=0$), so we skip some
  details.  Consider the following link diagrams for $K$, $K'$, and
  $K\# K'$, so that the induced diagram for $K\amalg K'$ differs from
  the diagram of $K\# K'$ locally by an elementary saddle.
  \[
  \vcenter
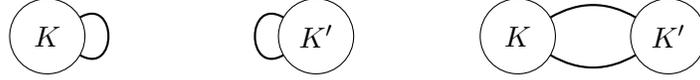
{\hbox{
  \begin{tikzpicture}
    \draw[knot] (-30:0.5) to[looseness=3,out=-30,in=30] (30:0.5);
    \node[circle,fill=white,draw,inner sep=0,outer sep=0,minimum width=1cm] at (0,0) {$K$};
  \end{tikzpicture}}}
  \qquad\qquad
  \vcenter{\hbox{
  \begin{tikzpicture}
    \draw[knot] (150:0.5) to[looseness=3,out=150,in=-150] (-150:0.5);
    \node[circle,fill=white,draw,inner sep=0,outer sep=0,minimum width=1cm] at (0,0) {$K'$};
  \end{tikzpicture}}}
  \qquad\qquad
  \vcenter{\hbox{
  \begin{tikzpicture}
    \begin{scope}[xshift=-2cm]
      \node[circle,fill=white,draw,inner sep=0,outer sep=0,minimum width=1cm] at (0,0) {$K$};
      \coordinate (t) at (30:0.5);
      \coordinate (b) at (-30:0.5);
    \end{scope}
    \node[circle,fill=white,draw,inner sep=0,outer sep=0,minimum width=1cm] at (0,0) {$K'$};
    \draw[knot] (t) to[out=30,in=150] (150:0.5);
    \draw[knot] (b) to[out=-30,in=-150] (-150:0.5);
  \end{tikzpicture}}}
  \]
  Let $\KhCxL$ be the Khovanov chain complexes associated to these
  diagrams. They become modules over $R[X]$ by the $X$-action at the
  strands that are shown in the above diagram. (For $K\#K'$ either
  strand works.) By construction, these complexes are free over
  $R[T]$, but indeed, they are free over $R[X]$ as well. Therefore, it
  is enough to construct an isomorphism of chain complexes over
  $R[X]$,
  \[
  \KhCxL(K\# K')\cong\Sigma^{0,-1}\KhCxL(K)\otimes_{R[X]}\KhCxL(K').
  \]

  Consider the saddle map
  \[
  \KhCxL(K)\otimes_{R[T]}\KhCxL(K')\cong\KhCxL(K\amalg K')
  \to\Sigma^{0,1}\KhCxL(K\# K'),
  \]
  and it is easy to check that it factors through
  $\KhCxL(K)\otimes_{R[X]}\KhCxL(K')$. So all that remains is to check
  that this $R[X]$-module chain map 
  \[
  \KhCxL(K)\otimes_{R[X]}\KhCxL(K')\to\Sigma^{0,1}\KhCxL(K\# K')
  \]
  is an isomorphism on the chain groups.

  The chain groups $\KhCxL$ decompose as direct sums of chain groups
  of various resolutions of the link diagrams, so it is enough to
  check that the above map is an isomorphism at each resolution of $K$
  and $K'$---that is, it is enough to check the case when $K$ and $K'$
  are planar unlinks, which is trivial to check.
\end{proof}

If $K$ is a knot, and $R$ is a field $\FF$ with $2\neq 0$, then the
module $\KhL(K)$ over $\FF[X]$ takes a particularly simple form. It
decomposes (non-canonically) as $\Sigma^{0,s(K)+1}\FF[X]\oplus
\XTorsion(K)$, where $s(K)$ is Rasmussen's $s$-invariant, and
$\XTorsion(K)$ is the (canonical) subgroup of $\KhL(K)$ consisting of
the $X$-torsion elements,
\[
\XTorsion(K)=\{a\in\KhL(K)\mid \exists n, X^na=0\},
\]
which we will call the \emph{extortion group} of $K$.

The smallest $n$ such that $X^n\XTorsion(K)=0$ is called the
\emph{extortion order}, and denoted $\XOrder(K)$. This was used
earlier in~\cite{AD-kh-unknotting} to provide a lower bound on the
unknotting number. The extortion order $\XOrder(K)$ is related to the
Lee spectral sequence (coming from the filtered chain complex for
$\KhL(K)$ with filtration given by powers of $T$) as follows. If the
Lee spectral sequence collapses at the $E_k$ page, then
$\XOrder(K)\in\{2k-3,2k-2\}$.

The only knot with $\XOrder(K)=0$ (that is, $\XTorsion(K)=0$) is the
unknot~\cite{KM-kh-unknot}. All other Kh-thin knots have
$\XOrder(K)=1$; $8_{19}$ is the first knot with $\XOrder(K)=2$. Since
the Lee spectral sequence collapses at the $E_2$ page for small knots,
it is hard to find examples of knots with $\XOrder(K)>2$; the first
example of a knot with $\XOrder(K)>2$ was constructed
in~\cite{MM-kh-knight}.

The extortion order can be computed from the Mathematica package
KnotTheory~\cite{KAT-kh-knotatlas} using the function UniversalKh, the
standard reference for which seems to be `Scott's
slides'~\cite{Sco-kh-slides}. UniversalKh works over $\QQ$ and returns a free
resolution of $\KhL(K)$ over $\QQ[X]$. Each term $t^aq^b\mathrm{KhE}$
contributes a tower $\QQ[X]\langle p\rangle$ with the generator $p$ in
bigrading $(a,b)$, and each term $t^aq^b\mathrm{KhC[n]}$ contributes a
two-step complex $\QQ[X]\langle p,q\rangle$ with generators $p,q$ in
bigradings $(a-1,b-2n),(a,b)$ and differential $p\mapsto X^n
q$. Therefore, the extortion group $\XTorsion(K)$ over $\QQ$ is
isomorphic to the homology of complexes coming from the
$\mathrm{KhC}[n]$ terms and the extortion order $\XOrder(K)$ over
$\QQ$ is the largest $n$ so that $\mathrm{KhC}[n]$ appears.

The extortion groups and extortion orders behave nicely under connect
sums.
\begin{lemma}\label{lem:xo-connect-sum}
  Consider knots $K,K'$ and their connect sum $K\# K'$. Then over any
  field $\FF$ with $2\neq 0$,
  \[
  \XTorsion(K\# K')\cong\Sigma^{0,s(K')}
  \XTorsion(K)\oplus\Sigma^{0,s(K)}\XTorsion(K')\oplus
  \Sigma^{0,-1}\Tor_{\FF[X]}(\XTorsion(K),\XTorsion(K')),
  \]
  and
  \[
  \XOrder(K\# K')=\max\{\XOrder(K),\XOrder(K')\}.
  \]
\end{lemma}

\begin{proof}
  The first statement is immediate from
  Lemma~\ref{lem:khovanov-connect-sum} and the isomorphism
  \[
  \KhL(L)\cong \Sigma^{0,s(L)+1}\FF[X]\oplus \XTorsion(L)
  \] for all knots $L$.

  For the second statement, we immediately get 
  \[
  \XOrder(K\# K')\geq \max\{\XOrder(K),\XOrder(K')\}
  \] 
  from the first two summands in the decomposition of $\XTorsion(K\#
  K')$. So it is enough to prove that the extortion order of the
  summand $\Tor_{\FF[X]}(\XTorsion(K),\XTorsion(K'))$ equals the
  minimum of the extortion order of $\XTorsion(K)$ and
  $\XTorsion(K')$.

  Consider free resolutions $\wt{\XTorsion}(K)$ and
  $\wt{\XTorsion}(K')$ of the extortion groups over $\FF[X]$. By the
  classification of finitely generated modules over PID's, they
  decompose into a direct sum of 2-step complexes $\FF[X]\langle
  p,q\rangle$, with the differential given by $p\mapsto \alpha(X)q$,
  where $\alpha(X)$ is some power of some irreducible homogeneous
  polynomial in $X$. Since $X$ has non-zero bigrading, the only
  possibilities are $\alpha(X)=X^n$. Each such summand contributes
  $\FF[X]\langle q\rangle/\{X^nq=0\}$ in homology, so the extortion
  orders are the maximum $n$'s that appear in such a decomposition.

  If $n\geq m$, then by a simple change of basis, the tensor product
  of the 2-step complexes $\FF[X]\stackrel{X^n}{\too}\FF[X]$ and
  $\FF[X]\stackrel{X^m}{\too}\FF[X]$ decomposes as
  $\big(\FF[X]\stackrel{X^m}{\too}\FF[X]\big)\oplus
  \big(\FF[X]\stackrel{X^m}{\too}\FF[X]\big)$, and hence, the
  extortion order of the summand
  $\Tor_{\FF[X]}(\XTorsion(K),\XTorsion(K'))$ equals
  $\min\{\XTorsion(K),\XTorsion(K')\}$.
\end{proof}

\section{Main theorem}\label{sec:main}
This section is devoted to the proof of the main theorems from
Section~\ref{sec:introduction}.

\begin{proof}[Proof of Theorem~\ref{thm:main}]
  Since the ribbon distance is defined using a sequence of ribbon
  concordances, it is enough to do the case when there is a ribbon
  concordance $K\stackrel{F}{\too} K'$ with at most $d$ saddles. After
  isotopy, we assume the movie of the cobordism $F$ has the following
  form.
  \begin{enumerate}[leftmargin=*]
  \item First we perform some Reidemeister moves and planar isotopy on
    $K$. Since we are free to choose the link diagram for $K$, we
    actually do not need this move.
  \item Then we perform $d$ elementary (planar) saddles, one at a time.
  \item Then we perform further Reidemeister moves and planar isotopy.
  \item Then we perform $d$ elementary (planar) deaths, again one at a
    time.
  \item Then we again perform some Reidemeister moves and planar
    isotopy to end at $K'$. Once again, since we are free to choose
    the link diagram for $K'$, we do not need this move.
  \end{enumerate}
  So the ribbon concordance $K\stackrel{F}{\too} K'$ decomposes as
  $K\stackrel{F_2}{\too} \wt{K}\stackrel{F_3}{\too} K'\amalg
  U^d\stackrel{F_4}{\too} K'$, where the piece $F_i$ comes from
  Item-($i$) above, and $U^d$ denotes $d$-component planar unlink.

  Let $K'\stackrel{\ol{F}}{\too} K$ be the cobordism viewed in reverse
  (which decomposes as $K'\stackrel{\ol{F}_4}{\too}K'\amalg U^d
  \stackrel{\ol{F}_3}{\too} \wt{K}\stackrel{\ol{F}_2}{\too} K$). Let
  $K\stackrel{W}{\too} K$ be the cobordism
  \[
  K\stackrel{F_2}{\too} \wt{K}\stackrel{F_3}{\too} K'\amalg
  U^d\stackrel{\ol{F}_3}{\too} \wt{K}\stackrel{\ol{F}_2}{\too} K.
  \]
  We will prove Theorem~\ref{thm:main} by computing the image of the
  map $\KhL(W)$ in two different ways, corresponding to the two sides
  of the equation in the statement of the theorem.

  \begin{itemize}[leftmargin=*]
  \item\textbf{Method 1.} The cobordism $\wt{K}\stackrel{F_3}{\too}
    K'\amalg U^d\stackrel{\ol{F}_3}{\too} \wt{K}$ is isotopic (rel
    boundary) to the identity cobordism $\wt{K}\to\wt{K}$ since the
    cobordism $F_3$ corresponds to a link isotopy in $\RR^3$, and
    $\ol{F}_3$ is the same isotopy performed in reverse. Therefore,
    the image of $\KhL(W)$ is same as the image of the map associated
    to the cobordism $K\stackrel{F_2}{\too}
    \wt{K}\stackrel{\ol{F}_2}{\too} K$.

    This cobordism performs $d$ planar saddles, and then performs them
    in reverse. So repeated applications of
    Lemma~\ref{lem:neck-cutting-2} tells us that the map
    $\KhL(\ol{F}_2)\circ\KhL(F_2)$ associated to this cobordism, up to
    an overall sign, is $2^d$ times the map associated to the dotted
    cobordism $K\stackrel{P}{\too} K$, where $P$ is the product
    cobordism decorated with $d$ dots. (Note, since $P$ is connected,
    by Lemma~\ref{lem:dotted-map-same}, the map $\KhL(P)$ is
    independent of the placement of the $d$ dots on $P$.) By
    definition, the image of $\KhL(P)$ is $X^d\KhL(K)$; therefore, the
    image of the original cobordism map is $(2X)^d\KhL(K)$.

    Schematically (with $d=1$):
    \[
    \vcenter{\hbox{
        \begin{tikzpicture}[rotate=90,yscale=0.3]
          \draw[knot] (0,6) circle (1);
          \node[anchor=north] at (-1,6) {$K$};
          
          \draw[knot] (-1,3) arc (-180:0:0.5);
          \draw[knot] (0.2,3) arc (-180:0:0.4);
          \node[anchor=north] at (-1,3) {$\wt{K}$};
          
          \draw[knot] (-1.5,0) arc (-180:0:0.8);
          \draw[knot] (0.8,0) arc (-180:0:0.3);
          \node[anchor=north] at (-1.5,0) {$K'$};
          \node[anchor=north] at (0.8,0) {$U$};
          
          \draw[knot] (-1,-3) arc (-180:0:0.5);
          \draw[knot] (0.2,-3) arc (-180:0:0.4);
          \node[anchor=north] at (-1,-3) {$\wt{K}$};
          
          \draw[knot] (-1,-6) arc (-180:0:1);
          \node[anchor=north] at (-1,-6) {$K$};

          \draw (-1,6) -- (-1,3);
          \draw (0,3) to[looseness=10,out=90,in=90] (0.2,3);
          \draw (1,6) -- (1,3);
          \node at (-0.5,4) {$F_2$};
          
          \draw (-1,3) to[looseness=1,out=-90,in=90] (-1.5,0);
          \draw (0,3) to[looseness=1,out=-90,in=90] (0.1,0);
          \draw (0.2,3) to[looseness=1,out=-90,in=90] (0.8,0);
          \draw (1,3) to[looseness=1,out=-90,in=90] (1.4,0);
          \node at (-0.5,1) {$F_3$};
          
          \draw (-1,-3) to[looseness=1,out=90,in=-90] (-1.5,0);
          \draw (0,-3) to[looseness=1,out=90,in=-90] (0.1,0);
          \draw (0.2,-3) to[looseness=1,out=90,in=-90] (0.8,0);
          \draw (1,-3) to[looseness=1,out=90,in=-90] (1.4,0);
          \node at (-0.5,-2) {$\ol{F}_3$};
          
          \draw (-1,-6) -- (-1,-3);
          \draw (0,-3) to[looseness=10,out=-90,in=-90] (0.2,-3);
          \draw (1,-6) -- (1,-3);
          \node at (-0.5,-5) {$\ol{F}_2$};
        \end{tikzpicture}}}
    =(-1)^\alpha
    \vcenter{\hbox{
        \begin{tikzpicture}[rotate=90,yscale=0.3]
          \draw[knot] (0,6) circle (1);
          \node[anchor=north] at (-1,6) {$K$};
          
          \draw[knot] (-1,3) arc (-180:0:0.5);
          \draw[knot] (0.2,3) arc (-180:0:0.4);
          \node[anchor=north] at (-1,3) {$\wt{K}$};
          
          \draw[knot] (-1,0) arc (-180:0:1);
          \node[anchor=north] at (-1,0) {$K$};

          \draw (-1,6) -- (-1,3);
          \draw (0,3) to[looseness=10,out=90,in=90] (0.2,3);
          \draw (1,6) -- (1,3);
          \node at (-0.5,4) {$F_2$};
          
          \draw (-1,0) -- (-1,3);
          \draw (0,3) to[looseness=10,out=-90,in=-90] (0.2,3);
          \draw (1,0) -- (1,3);
          \node at (-0.5,1) {$\ol{F}_2$};
        \end{tikzpicture}}}
    =2(-1)^\beta
    \vcenter{\hbox{
        \begin{tikzpicture}[rotate=90,yscale=0.3]
          \draw[knot] (0,3) circle (1);
          \node[anchor=north] at (-1,3) {$K$};
          
          \draw[knot] (-1,0) arc (-180:0:1);
          \node[anchor=north] at (-1,0) {$K$};

          \draw (-1,0) -- (-1,3);
          \draw (1,0) -- (1,3);
          \node at (-0.5,1) {$P$};
          \node[dot] at (0,0.6) {};
        \end{tikzpicture}}}
    \]
  \item\textbf{Method 2.} Insert the identity cobordism $K'\amalg
    U^d\to K'\amalg U^d$ into $W$ to get an isotopic cobordism
    \[
    K\stackrel{F_2}{\too} \wt{K}\stackrel{F_3}{\too} K'\amalg
    U^d\stackrel{\Id}{\too}K'\amalg U^d\stackrel{\ol{F}_3}{\too}
    \wt{K}\stackrel{\ol{F}_2}{\too} K.
    \]
    The cobordism $K'\amalg U^d\stackrel{\Id}{\too}K'\amalg U^d$ has
    $d$ necks coming from $U^d$, so repeated applications of
    Lemma~\ref{lem:neck-cutting-1} tells us that the map
    $\Id\from\KhL(K'\amalg U^d)\to\KhL(K'\amalg U^d)$, up to an
    overall sign, is the sum of $2^d$ maps associated to the following
    $2^d$ dotted cobordisms with $d$ dots: Each has the same
    underlying surface $K'\amalg U^d\stackrel{F_4}{\too}K'
    \stackrel{\ol{F}_4}{\too}K'\amalg U^d$ which has $d$ death-birth
    pairs; and the $2^d$ dotted cobordisms are obtained by
    distributing $d$ dots in $2^d$ different ways so that each
    death-birth pair has exactly one dot.

    The underlying composed cobordism
    \[
    \big(K\stackrel{F_2}{\too} \wt{K}\stackrel{F_3}{\too} K'\amalg
    U^d\stackrel{F_4}{\too}K'\stackrel{\ol{F}_4}{\too}K'\amalg
    U^d\stackrel{\ol{F}_3}{\too} \wt{K}\stackrel{\ol{F}_2}{\too} K\big)=
    K\stackrel{F}{\too} K'\stackrel{\ol{F}}{\too}K
    \]
    is connected, so by Lemma~\ref{lem:dotted-map-same}, the $2^d$
    dotted cobordism maps all induce the same map, which is the map
    $\KhL(\ol{F})\circ\KhL(Q)\circ\KhL(F)$ corresponding to the dotted
    cobordism $K\stackrel{F}{\too}
    K'\stackrel{Q}{\too}K'\stackrel{\ol{F}}{\too}K$, where
    $K'\stackrel{Q}{\too}K'$ is the product cobordism decorated with
    $d$ dots.

    Levine and Zemke has shown~\cite{LZ-kh-ribbon} that the map
    $\KhL(F)\circ\KhL(\ol{F})\from \KhL(K')\to\KhL(K')$ is $\pm
    \Id$,\footnote{Actually, they proved it for Khovanov's
      specialization $X^2=T=0$, but the proof works in this more
      general case.}  and therefore, the map $\KhL(F)$ is surjective
    and the map $\KhL(\ol{F})$ is injective. Consequently, the image
    of the map $\KhL(\ol{F})\circ\KhL(Q)\circ\KhL(F)$ is isomorphic to
    the image of $\KhL(Q)$, which is $X^d\KhL(K')$. Therefore, the
    image of the original cobordism map is isomorphic to
    $(2X)^d\KhL(K')$.

    Schematically:
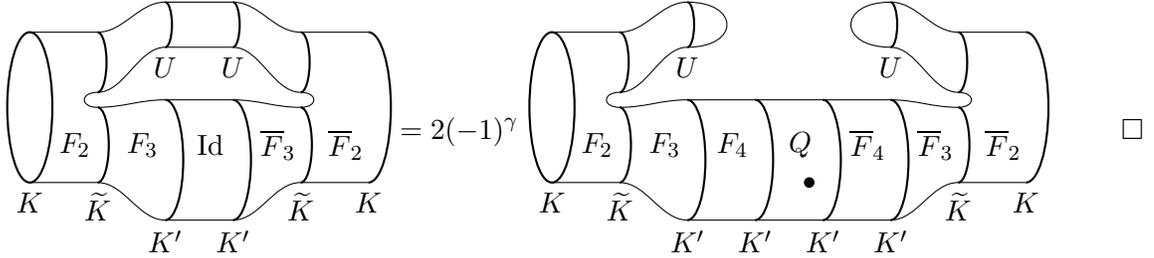
    \[
    \vcenter{\hbox{
        \begin{tikzpicture}[rotate=90,yscale=0.3]
          \draw[knot] (0,9) circle (1);
          \node[anchor=north] at (-1,9) {$K$};
          
          \draw[knot] (-1,6) arc (-180:0:0.5);
          \draw[knot] (0.2,6) arc (-180:0:0.4);
          \node[anchor=north] at (-1,6) {$\wt{K}$};
          
          \draw[knot] (-1.5,3) arc (-180:0:0.8);
          \draw[knot] (0.8,3) arc (-180:0:0.3);
          \node[anchor=north] at (-1.5,3) {$K'$};
          \node[anchor=north] at (0.8,3) {$U$};

          \draw[knot] (-1.5,0) arc (-180:0:0.8);
          \draw[knot] (0.8,0) arc (-180:0:0.3);
          \node[anchor=north] at (-1.5,0) {$K'$};
          \node[anchor=north] at (0.8,0) {$U$};
          
          \draw[knot] (-1,-3) arc (-180:0:0.5);
          \draw[knot] (0.2,-3) arc (-180:0:0.4);
          \node[anchor=north] at (-1,-3) {$\wt{K}$};
          
          \draw[knot] (-1,-6) arc (-180:0:1);
          \node[anchor=north] at (-1,-6) {$K$};

          \draw (-1,9) -- (-1,6);
          \draw (0,6) to[looseness=10,out=90,in=90] (0.2,6);
          \draw (1,9) -- (1,6);
          \node at (-0.5,7) {$F_2$};
          
          \draw (-1,6) to[looseness=1,out=-90,in=90] (-1.5,3);
          \draw (0,6) to[looseness=1,out=-90,in=90] (0.1,3);
          \draw (0.2,6) to[looseness=1,out=-90,in=90] (0.8,3);
          \draw (1,6) to[looseness=1,out=-90,in=90] (1.4,3);
          \node at (-0.5,4) {$F_3$};
          
          \draw (-1.5,3) -- (-1.5,0);
          \draw (0.1,3) -- (0.1,0);
          \draw (0.8,3) -- (0.8,0);
          \draw (1.4,3) -- (1.4,0);
          \node at (-0.5,1) {$\Id$};

          \draw (-1,-3) to[looseness=1,out=90,in=-90] (-1.5,0);
          \draw (0,-3) to[looseness=1,out=90,in=-90] (0.1,0);
          \draw (0.2,-3) to[looseness=1,out=90,in=-90] (0.8,0);
          \draw (1,-3) to[looseness=1,out=90,in=-90] (1.4,0);
          \node at (-0.5,-2) {$\ol{F}_3$};
          
          \draw (-1,-6) -- (-1,-3);
          \draw (0,-3) to[looseness=10,out=-90,in=-90] (0.2,-3);
          \draw (1,-6) -- (1,-3);
          \node at (-0.5,-5) {$\ol{F}_2$};
        \end{tikzpicture}}}
    =2(-1)^\gamma
    \vcenter{\hbox{
        \begin{tikzpicture}[rotate=90,yscale=0.3]
          \draw[knot] (0,15) circle (1);
          \node[anchor=north] at (-1,15) {$K$};
          
          \draw[knot] (-1,12) arc (-180:0:0.5);
          \draw[knot] (0.2,12) arc (-180:0:0.4);
          \node[anchor=north] at (-1,12) {$\wt{K}$};
          
          \draw[knot] (-1.5,9) arc (-180:0:0.8);
          \draw[knot] (0.8,9) arc (-180:0:0.3);
          \node[anchor=north] at (-1.5,9) {$K'$};
          \node[anchor=north] at (0.8,9) {$U$};

          \draw[knot] (-1.5,6) arc (-180:0:0.8);
          \node[anchor=north] at (-1.5,6) {$K'$};

          \draw[knot] (-1.5,3) arc (-180:0:0.8);
          \node[anchor=north] at (-1.5,3) {$K'$};

          \draw[knot] (-1.5,0) arc (-180:0:0.8);
          \draw[knot] (0.8,0) arc (-180:0:0.3);
          \node[anchor=north] at (-1.5,0) {$K'$};
          \node[anchor=north] at (0.8,0) {$U$};
          
          \draw[knot] (-1,-3) arc (-180:0:0.5);
          \draw[knot] (0.2,-3) arc (-180:0:0.4);
          \node[anchor=north] at (-1,-3) {$\wt{K}$};
          
          \draw[knot] (-1,-6) arc (-180:0:1);
          \node[anchor=north] at (-1,-6) {$K$};

          \draw (-1,15) -- (-1,12);
          \draw (0,12) to[looseness=10,out=90,in=90] (0.2,12);
          \draw (1,15) -- (1,12);
          \node at (-0.5,13) {$F_2$};
          
          \draw (-1,12) to[looseness=1,out=-90,in=90] (-1.5,9);
          \draw (0,12) to[looseness=1,out=-90,in=90] (0.1,9);
          \draw (0.2,12) to[looseness=1,out=-90,in=90] (0.8,9);
          \draw (1,12) to[looseness=1,out=-90,in=90] (1.4,9);
          \node at (-0.5,10) {$F_3$};
          
          \draw (-1.5,9) -- (-1.5,6);
          \draw (0.1,9) -- (0.1,6);
          \draw (0.8,9) to[looseness=10,out=-90,in=-90] (1.4,9);
          \node at (-0.5,7) {$F_4$};

          \draw (-1.5,6) -- (-1.5,3);
          \draw (0.1,6) -- (0.1,3);
          \node at (-0.5,4) {$Q$};
          \node[dot] at (-1,3.6) {};

          \draw (-1.5,3) -- (-1.5,0);
          \draw (0.1,3) -- (0.1,0);
          \draw (0.8,0) to[looseness=10,out=90,in=90] (1.4,0);
          \node at (-0.5,1) {$\ol{F}_4$};

          \draw (-1,-3) to[looseness=1,out=90,in=-90] (-1.5,0);
          \draw (0,-3) to[looseness=1,out=90,in=-90] (0.1,0);
          \draw (0.2,-3) to[looseness=1,out=90,in=-90] (0.8,0);
          \draw (1,-3) to[looseness=1,out=90,in=-90] (1.4,0);
          \node at (-0.5,-2) {$\ol{F}_3$};
          
          \draw (-1,-6) -- (-1,-3);
          \draw (0,-3) to[looseness=10,out=-90,in=-90] (0.2,-3);
          \draw (1,-6) -- (1,-3);
          \node at (-0.5,-5) {$\ol{F}_2$};
        \end{tikzpicture}}}
    \qedhere
    \]
  \end{itemize}
\end{proof}

\begin{proof}[Proof of Corollary~\ref{cor:main}]
  Let $d$ be the ribbon distance of the knot $K$ to the unknot $U$. We
  know $\KhL(U)\cong\Sigma^{0,1}\FF[X]$ and since we are working over
  a field with $2\neq 0$,
  \[
  \KhL(K)\cong\Sigma^{0,s(K)+1}\FF[X]\oplus\XTorsion(K).
  \]
  (If $d<\infty$, then $K$ is slice, and hence
  $s(K)=0$~\cite{Ras-kh-slice}, but we will not need this fact;
  indeed, this fact will follow from the proof.)

  By Theorem~\ref{thm:main}, and since $2\neq 0$,
  \[
  X^d\KhL(K)\cong\Sigma^{0,s(K)+1-2d}\FF[X]\oplus X^d\XTorsion(K)\cong
  X^d\KhL(U)\cong\Sigma^{0,1-2d}\FF[X],
  \]
  and therefore, $X^d\XTorsion(K)=0$, and hence, $d\geq \XOrder(K)$.
\end{proof}

\begin{example}\label{exam:d-is-2}
  Let $K$ be the connect sum of $8_{19}$ and its mirror. Using the
  function UniversalKh from the Mathematica package KnotTheory, we get
  $\XOrder(8_{19})=2$.  By Lemma~\ref{lem:xo-connect-sum},
  $\XOrder(K)\geq 2$, and hence by Corollary~\ref{cor:main}, the
  distance of $K$ from the unknot is at least $2$. Indeed, adding
  untwisted (blackboard-framed) bands along the thick lines in the
  following knot diagram for $K$ converts it to a $3$-component
  unlink, so $d(K,U)=2$.
  
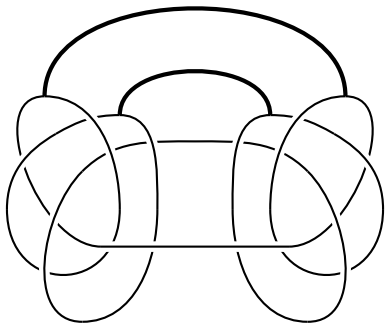
\[
  \begin{tikzpicture}[scale=0.5]
    \coordinate (m0) at (0,1.8);
    \coordinate (m1) at (0,-1);

    \coordinate (a0) at (1,0);
    \coordinate (a1) at (2,0);
    \coordinate (a2) at (3,1);
    \coordinate (a3) at (5,0);
    \coordinate (a4) at (3,-3);
    \coordinate (a5) at (4,3);
    \coordinate (a6) at (2.5,-1);
    \coordinate (a7) at (2,2.5);
    \coordinate (a8) at (3.7,2);

    \coordinate (b0) at (-1,0);
    \coordinate (b1) at (-2,0);
    \coordinate (b2) at (-3,1);
    \coordinate (b3) at (-5,0);
    \coordinate (b4) at (-3,-3);
    \coordinate (b5) at (-4,3);
    \coordinate (b6) at (-2.5,-1);
    \coordinate (b7) at (-2,2.5);
    \coordinate (b8) at (-3.7,2);

    \draw[knott] (m0) to[out=0,in=135] (a2);
    \draw[knott] (a4) to[out=180,in=-90] (a0);
    \draw[knott] (a5) to[out=0,in=0] (a6);
    \draw[knott] (a3) to[looseness=2,out=-90,in=-90] (a1);
    \draw[knott] (a7) to[out=0,in=150] (a8);

    \draw[knott] (a8) to[out=-30,in=90] (a3);
    \draw[knott] (a6) to[out=180,in=0] (m1);
    \draw[knott] (a2) to[out=-45,in=0] (a4);
    \draw[knott] (a0) to[out=90,in=180] (a7);
    \draw[knott] (a1) to[out=90,in=180] (a5);

    \draw[knott] (m0) to[out=180,in=45] (b2);
    \draw[knott] (b4) to[out=0,in=-90] (b0);
    \draw[knott] (b5) to[out=180,in=180] (b6);
    \draw[knott] (b3) to[looseness=2,out=-90,in=-90] (b1);
    \draw[knott] (b7) to[out=180,in=30] (b8);

    \draw[knott] (b8) to[out=-150,in=90] (b3);
    \draw[knott] (b6) to[out=0,in=180] (m1);
    \draw[knott] (b2) to[out=-135,in=180] (b4);
    \draw[knott] (b0) to[out=90,in=0] (b7);
    \draw[knott] (b1) to[out=90,in=0] (b5);

    \draw[ultra thick] (a7) to[out=90,in=90] (b7);
    \draw[ultra thick] (a5) to[out=90,in=90] (b5);

  \end{tikzpicture}
  \]
\end{example}

\begin{example}\label{exam:d-is-3}
  Let $K_M$ be the knot from~\cite{MM-kh-knight}, and let $K$ be the
  connect sum of $K_M$ and its mirror. Since the Lee spectral sequence
  for $K_M$ collapses at the $E_3$ page, we know $\XOrder(K_M)\geq 3$.
  Once again, by Lemma~\ref{lem:xo-connect-sum}, $\XOrder(K)\geq 3$,
  and hence by Corollary~\ref{cor:main}, the distance of $K$ from the
  unknot is at least $3$.
\end{example}

\bibliography{ribbonbibfile}

\providecommand{\bysame}{\leavevmode\hbox to3em{\hrulefill}\thinspace}
\providecommand{\MR}{\relax\ifhmode\unskip\space\fi MR }
\providecommand{\MRhref}[2]{%
  \href{http://www.ams.org/mathscinet-getitem?mr=#1}{#2}
}
\providecommand{\href}[2]{#2}
\begin{thebibliography}{CMW09}

\bibitem[AD]{AD-kh-unknotting}
Akram Alishahi and Nathan Dowlin, \emph{The {L}ee spectral sequence, unknotting
  number, and the knight move conjecture}, arXiv:1710.07875.

\bibitem[Bar05]{Bar-kh-tangle-cob}
Dror Bar-Natan, \emph{Khovanov's homology for tangles and cobordisms}, Geom.
  Topol. \textbf{9} (2005), 1443--1499. \MR{2174270 (2006g:57017)}

\bibitem[BLS17]{BLS-kh-pointed}
John~A. Baldwin, Adam~Simon Levine, and Sucharit Sarkar, \emph{Khovanov
  homology and knot {F}loer homology for pointed links}, J. Knot Theory
  Ramifications \textbf{26} (2017), no.~2, 1740004, 49. \MR{3604486}

\bibitem[BM]{KAT-kh-knotatlas}
Dror Bar-Natan, Scott Morrison, and et~al., \emph{The {K}not {A}tlas},
  \url{http://katlas.org/}.

\bibitem[CMW09]{CMW-kh-functoriality}
David Clark, Scott Morrison, and Kevin Walker, \emph{Fixing the functoriality
  of {K}hovanov homology}, Geom. Topol. \textbf{13} (2009), no.~3, 1499--1582.
  \MR{2496052}

\bibitem[CRS97]{CRS-knot-knottedsurfaces}
J.~Scott Carter, Joachim~H. Rieger, and Masahico Saito, \emph{A combinatorial
  description of knotted surfaces and their isotopies}, Adv. Math. \textbf{127}
  (1997), no.~1, 1--51. \MR{1445361 (98c:57023)}

\bibitem[Fox73]{Fox-top-ribbon}
Ralph~H. Fox, \emph{Characterizations of slices and ribbons}, Osaka J. Math.
  \textbf{10} (1973), 69--76. \MR{0326705}

\bibitem[Gor81]{Gor-top-ribbon}
C.~McA. Gordon, \emph{Ribbon concordance of knots in the {$3$}-sphere}, Math.
  Ann. \textbf{257} (1981), no.~2, 157--170. \MR{634459}

\bibitem[Jac04]{Jac-kh-cobordisms}
Magnus Jacobsson, \emph{An invariant of link cobordisms from {K}hovanov
  homology}, Algebr. Geom. Topol. \textbf{4} (2004), 1211--1251 (electronic).
  \MR{2113903 (2005k:57047)}

\bibitem[Kan10]{Kan-top-band}
Taizo Kanenobu, \emph{Band surgery on knots and links}, J. Knot Theory
  Ramifications \textbf{19} (2010), no.~12, 1535--1547. \MR{2755489}

\bibitem[Kho00]{Kho-kh-categorification}
Mikhail Khovanov, \emph{A categorification of the {J}ones polynomial}, Duke
  Math. J. \textbf{101} (2000), no.~3, 359--426. \MR{1740682 (2002j:57025)}

\bibitem[Kho02]{Kho-kh-tangles}
\bysame, \emph{A functor-valued invariant of tangles}, Algebr. Geom. Topol.
  \textbf{2} (2002), 665--741 (electronic). \MR{1928174 (2004d:57016)}

\bibitem[Kho06]{Kho-kh-cobordism}
\bysame, \emph{An invariant of tangle cobordisms}, Trans. Amer. Math. Soc.
  \textbf{358} (2006), no.~1, 315--327. \MR{2171235 (2006g:57046)}

\bibitem[KM11]{KM-kh-unknot}
P.~B. Kronheimer and T.~S. Mrowka, \emph{Khovanov homology is an
  unknot-detector}, Publ. Math. Inst. Hautes \'{E}tudes Sci. (2011), no.~113,
  97--208. \MR{2805599}

\bibitem[KR08]{KR-kh-matrixfactorizations}
Mikhail Khovanov and Lev Rozansky, \emph{Matrix factorizations and link
  homology}, Fund. Math. \textbf{199} (2008), no.~1, 1--91. \MR{2391017
  (2010a:57011)}

\bibitem[Lee05]{Lee-kh-endomorphism}
Eun~Soo Lee, \emph{An endomorphism of the {K}hovanov invariant}, Adv. Math.
  \textbf{197} (2005), no.~2, 554--586. \MR{2173845 (2006g:57024)}

\bibitem[LZ]{LZ-kh-ribbon}
Adam~Simon Levine and Ian Zemke, \emph{Khovanov homology and ribbon
  concordance}, arXiv:1903.01546.

\bibitem[MM]{MM-kh-knight}
Ciprian Manolescu and Marco Marengon, \emph{The knight move conjecture is
  false}, arXiv:1809.09769.

\bibitem[Mor]{Sco-kh-slides}
Scott Morrison, \emph{Universal {K}hovanov homology}, http://tqft.net/kyoto2.

\bibitem[NN82]{NN-top-ribbon}
Yasutaka Nakanishi and Yoko Nakagawa, \emph{On ribbon knots}, Math. Sem. Notes
  Kobe Univ. \textbf{10} (1982), no.~2, 423--430. \MR{704925}

\bibitem[Pic]{Pic-kh-conway}
Lisa Piccirillo, \emph{The {C}onway knot is not slice}, arXiv:1808.02923.

\bibitem[Ras10]{Ras-kh-slice}
Jacob Rasmussen, \emph{Khovanov homology and the slice genus}, Invent. Math.
  \textbf{182} (2010), no.~2, 419--447. \MR{2729272 (2011k:57020)}

\end{thebibliography}
\bibliographystyle{amsalpha}

\end{document}